\documentclass{amsart}

\usepackage{amsmath,amssymb}
\usepackage{mathrsfs}
\usepackage{comment}

\usepackage[english]{babel}
\usepackage[utf8]{inputenc}
\usepackage[T1]{fontenc}
\usepackage{tikz-cd}

\usepackage[top=3cm,bottom=2cm,left=3cm,right=3cm,marginparwidth=1.75cm]{geometry}

\usepackage{graphicx}
\usepackage[colorinlistoftodos]{todonotes}
\usepackage[colorlinks=true, allcolors=blue]{hyperref}
\usepackage{cite}
\usepackage{enumitem}
\setlist[1]{itemsep=3pt}
\usepackage[all,cmtip]{xy}
\usepackage{tikz-cd}
\usepackage[capitalise]{cleveref}
\usepackage{autonum}
\usepackage{setspace}

\usepackage{xcolor}

\newcommand{\f}{\varphi}
\newcommand{\grad}{\mathrm{grad}}

\newcommand{\R}{\mathbb{R}}
\newcommand{\Po}{P}

\newcommand{\C}{\mathbb{C}}
\newcommand{\N}{\mathbb{N}}
\newcommand{\Z}{\mathbb{Z}}
\newcommand{\RP}{\mathbb{R}\mathrm{P}}

\newtheorem{theorem}{Theorem}
\newtheorem{proposition}[theorem]{Proposition}

\newtheorem{lemma}[theorem]{Lemma}
\newtheorem{definition}[theorem]{Definition}

\theoremstyle{remark} 
\newtheorem{remark}[theorem]{Remark}

\newcommand{\be}{\begin{equation}}
\newcommand{\ee}{\end{equation}}

\title{The probabilistic method in real singularity theory}
\author{Antonio Lerario and Michele Stecconi}
\date{\today}

\begin{document}

\maketitle
\begin{abstract}We explain how to use the probabilistic method to prove the existence of real polynomial singularities with rich topology, i.e. with total Betti number of the maximal possible order. We show how similar ideas can be used to produce real algebraic projective hypersurfaces with a rich structure of umbilical points.
\end{abstract}

\section{Introduction}
The probabilistic method is a nonconstructive method for proving the existence of a prescribed kind of mathematical object. It works by showing that if one randomly chooses objects from a specified class, the probability that the result is of the prescribed kind is strictly greater than zero, and in particular it guarantees the existence of one object with the desired property. 

In the context of real algebraic geometry, this method has been implicitly used first in \cite{NazarovSodin1} and then in \cite{SHSP, GaWe1} for computing a lower bound for the expectation (with respect to an appropriate probability measure) of the Betti numbers of random real hypersurfaces (a probabilistic result), and explicitly used in a similar framework in \cite{systoles}, for proving that complex hypersurfaces contain topologically rich Lagrangian submanifolds, and in \cite{Ancona} (using a different probability measure),  for proving the existence of real algebraic hypersurfaces with rich topology in real algebraic varieties (two deterministic results, see \cref{remark:H16} below).

In this short note we briefly explain  how to apply the method from \cite{Ancona} (which is based on the techniques that we have developed in \cite{DTGRF, MTTRPS}) more generally for the  construction of real algebraic \emph{singularities} with rich topology. More precisely, consider a polynomial map
$$p:S^n\to \R^k$$ all of whose components are homogeneous polynomials $p_1, \ldots, p_k$ of degree $d$.  For $r\geq0$ let also $J^r(S^n,\R^k)$ be the $r$--th jet bundle and $W\subseteq J^r(S^n,\R^k)$ be a semialgebraic subset invariant under diffeomorphisms of $S^n$ (we call sets of this type ``intrinsic'', see \cref{sec:proof}). We denote by 
$$j^rp:S^n\to J^{r}(S^n,\R^k)$$ 
the $r$--th jet extension of $p$ and we consider the set
$$\Sigma^W(p):=j^rp^{-1}(W)\subseteq S^n.$$
We call such a set \emph{the $W$--type singularity of $p$}. Examples of singularities falling in this class are: zero sets of polynomial functions, their critical points, Thom--Boardman singularities.

Using a variation on Thom--Milnor bound, it is not difficult to prove (see \cite[Theorem 1]{MTTRPS}) that for every $W\subseteq J^r(S^n, \R^k)$ there exists $c_1(W)>0$ such that for the generic polynomial map $p:S^n\to\R^k$ of degree $d$, the sum of the Betti numbers of its $W$--type singularity is controlled by
\be\label{eq:bound}b(\Sigma^W(p))\leq c_1(W)d^n.\ee

The aim of the current paper is to prove the following result.

\begin{theorem}\label{thm:main}For every nonempty intrinsic semialgebraic set $W\subseteq J^{r}(S^n, \R^k)$ with $0<\mathrm{codim}(W)\leq n$, there exist $d(W)\in \mathbb{N}$ and $c_2(W)>0$ such that for every $d\geq d(W)$ there is a nonempty open set $U_d^W$ in the space of polynomial maps of degree $d$ with the property that for every $p=(p_1, \ldots, p_k)\in U_d^W$ 
\be \label{eq:lower}b(\Sigma^W(p))\geq c_2(W)d^n.\ee
\end{theorem}

In this general case of maps singularities we are not aware of any other method, other than the one proposed in the current paper, for producing maps with rich singularities (of the maximal possible order).
\begin{remark}We will actually prove the stronger statement that one can find an open set $U_{d}^W$ in the space of polynomials such that for every $0\leq i\leq n-\mathrm{codim}(W)$  the inequality \eqref{eq:lower} holds for each single Betti number $b_i(\Sigma^W(p))$ for all $p\in U_{i,d}^W$. In other words, in Equation \eqref{eq:lower} we can replace $b(\Sigma^W(p))$ with $\underline{b}_{s}(\Sigma^W(p))$, where
\be 
\underline{b}_{s}(\Sigma):=\min_{0\le i\le s} b_i(\Sigma), \qquad s=n-\mathrm{codim}(W),
\ee
and the theorem remains true, with a possibly different constant $c_2(W)$.

The assumption $\mathrm{codim}(W)\leq n$ is necessary for our method of proof: in fact  the $r$--th jet of a generic polynomial map $p:S^n\to \R^k$ of large enough degree misses  $W$ if $\mathrm{codim}(W)>n$ (by the Parametric Transversality Theorem, see \cite[Section 1.7]{EliMis}).
\end{remark}

The proof of \cref{thm:main} uses a combination of results and ideas from \cite{DTGRF, MTTRPS, Ancona} and goes along the steps described in \cref{sec:proof}. In \cref{sec:umbilics} we will actually show that the same strategy of proof can be used to produce rich ``hybrid'' singularities, such as the structure of the umbilics of a hypersurface (which depends on the Riemannian metric also), see \cref{thm:umbilics}.
\subsection*{Acknwoledgements}The authors wish to thank R. Piene and B. Shapiro for interesting discussions and for bringing the problem to their attention. A. Lerario is supported by the Knuth and Alice Wallenberg Foundation. M. Stecconi was supported by the Luxembourg National Research Fund (Grant: 021/16236290/HDSA).
\section{Proof of \cref{thm:main}}\label{sec:proof}
\subsection{Realize the given singularity with a smooth map}\label{sec:step1} First recall from \cite{MTTRPS} that $W\subseteq J^{r}(S^n, \R^k)$ is called \emph{intrinsic} if there is  $W_0\subseteq J^r(n,k)$, called the \emph{model}, such that for any embedding $\varphi\colon \R^n\hookrightarrow S^n$, one has that $j^r\varphi^*(W)=W_0$.

Given $W$ intrinsic, with model $W_0$, we claim that there exists $f\in C^{\infty}(D, \R^k)$ such that (i) $j^rf$ is transversal to $W_0$, (ii) $\Sigma^{W_0}(f)$ is entirely contained in the interior of $D$ and (iii) $b_i(\Sigma^{W_0}(f))\geq 1$ for every $0\leq i\leq n-\mathrm{codim}(W)$. 

This guarantees that the singularity $W$ is smoothly realizable in a stable way with nontrivial cycles in  all dimensions. The existence of such $f$ is nontrivial and follows from \cite[Corollary 20]{MTTRPS}. 

In fact, it is easy to construct a section $\sigma$ of the jet bundle $J^{r}(n, k)\to \R^n$ which is transversal to $W_0$, with $\sigma^{-1}(W)\subset \mathrm{int}(D)$ and with $b_i(\sigma^{-1}(W_0))\geq 1$ for every $0\leq i\leq n-\mathrm{codim}(W_0)$, using the hypothesis on the codimension of $W_0$. However this section does not need to be holonomic. Using the Holonomic Approximation Theorem \cite{EliMis} one can find a {$C^0$ small homeomorphism $h:D\to D$, and a} new section $\widetilde{\sigma}$ \emph{holonomic}, i.e. such that $\widetilde{\sigma}=j^rf$, which is $C^0$ close to {$\sigma\circ h$, } so that $j^rf^{-1}(W)\subset \mathrm{int}(D)$, and which is transversal to $W_0$. Finally, \cite[Theorem 18]{MTTRPS} guarantees that $b_i(j^rf^{-1}(W_0))\geq b_i(\sigma^{-1}(W_0))$ for every $0\leq i\leq n-\mathrm{codim}(W_0)$.

\begin{remark}This step is not needed in the case $r=0$, since in this case sections are already holonomic.
\end{remark}

\subsection{Identify a nice open set in the space of functions}\label{sec:open}Let $f_0:D\to \R^k$ be the function constructed above, and denote by $\Sigma_0:=\Sigma^{W}(f_0)$. Since $j^rf_0$ is transversal to $W$, then there is a $C^{r+1}$ neighborhood $U_0$ of $f_0$ such that for any $g:D\to \R^k$ in $U_0$ (i.e for any $g$ sufficiently close to $f_0$ in the $C^{r+1}$--topology) the pairs $(D, \Sigma_0)$ and $(D, \Sigma^W(g))$ are diffeomorphic. This follows from Thom's Isotopy Lemma.  

%

\subsection{Choose an appropriate measure on the space of polynomials}\label{sec:step3}We choose an appropriate probability measure $\mu_d$ on the vector space 
$$\Po_{n,d}^k:=\R[x_0, \ldots, x_n]_{(d)}^k$$ of polynomial maps of degree $d$. This measure will be the Gaussian measure associated to the $L^2(S^n)$--scalar product. More precisely, for $p=(p_1, \ldots, p_k)\in \Po_{n,d}^k$ we define $\|p\|^2:=\int_{S^n}\sum_{i=1}^k |p_i(x)|^2 \mathrm{d}x,$
where ``$\mathrm{d}x$'' denotes the integration on the sphere with respect to its standard volume form. Then we define a Gaussian probability measure $\mu_d$ on $\Po_{n,d}^k$ by setting, for every Borel set $A\subseteq \Po_{n,d}^k$:
\be\label{eq:measure}\mu_d\bigg\{p\in A\bigg\}:=\frac{1}{(2\pi)^{N/2}}\int_{A}e^{-\frac{\|p\|^2}{2}}\mathrm{d}p,\ee
where $N=\mathrm{dim}(\Po_{n,d}^k)$ and ``$\mathrm{d}p$'' denotes the integration with respect to the Lebesgue measure\footnote{This is the Lebesgue measure of the Euclidean space $(\Po_{n,d}^k, \langle \cdot, \cdot \rangle_{L^2(S^n)})$.}. 
\begin{remark}In the context of real algebraic geometry, the measure $\mu_d$ was first used in \cite{SHSP}.  This measure is invariant under the action on the space of polynomials of the orthogonal group by change of variables, but there are other natural invariant measures one can consider \cite{Kostlan}. These measures, however, might not have the properties we need: for instance,  the Bombieri--Weyl measure would not produce rich objects, see \cref{sec:BW}. The key idea of \cite{Ancona} is to use a version of the above probability measure $\mu_d$ on sections of real line bundles.
\end{remark}
\subsection{Compute the limiting measure of the nice open set at the local scale}\label{sec:step4}Let $z\in S^n$ be a point and denote by $D:=D(0, 1)\subset T_{z}S^n\simeq \R^n$ the unit disk in the tangent space and by $\mathrm{exp}_z:D\to S^n$ the Riemannian exponential map. For every polynomial map $p:S^n\to \R^k$ we consider the smooth map ${f_{z,d}}:D\to \R^k$ defined by
\be \label{eq:rescaling}f_{z,d}(v):=d^{-\frac{n}{2}}p\left(\mathrm{exp}_z(d^{-1} v)\right).\ee
Since we have defined a probability measure $\mu_d$ on the space of polynomials, the map $f_{z,d}:D\to \R^k$ (which depends on $p$) can be thought as a random variable with values in $C^{\infty}(D, \R^k)$. 

The key property of this construction is the following: for every nonempty open set $U\subseteq C^{\infty}(D, \R^k)$ there exists $c(U)>0$ such that for every $z\in S^n$, 
\be \label{eq:lowerlimit}\liminf_{d\to \infty}\mu_d\bigg\{f_{z,d}\in U\bigg\}\geq c(U).\ee
\begin{proof}[Proof of \eqref{eq:lowerlimit}] We denote by $K_{d}(x,y):D\times D\to {\R^{k\times k}}$ the covariance function of $f_d:=f_{z,d}$ (which does not depend on $z\in S^n $, since the measure $\mu_d$ is orthogonally invariant):
$$K_{d}(u,v):=\mathbb{E}\left(f_d(u){f_d(v)^T}\right)\in \R^{k\times k},$$
where the expectation is taken with respect to the probability measure $\mu_d$.
It is a classical fact (see \cite{Hrmander1968TheSF}) that as $d\to \infty$ the function $K_d$ converges in the $C^{\infty}$--topology to
\be\label{eq:limitcovariance}
K_{\infty}(u,v)=\int_{D}e^{i\langle u-v , \xi\rangle}\mathrm{d}\xi \cdot \mathbf{1}_{k}.
\ee
Now, by \cite[Theorem 5]{DTGRF}, the convergence in the $C^{\infty}$--topology of $K_d$ implies the existence of a random variable $f_\infty\in C^{\infty}(D, \R^k)$ with the property that for every nonempty  open set $U\subseteq C^{\infty}(D, \R^k)$
\be\label{eq:defc(U)}\liminf_{d\to \infty}\mu_d\bigg\{f_d\in U\bigg\}\geq \mathbb{P}\bigg\{f_\infty\in U\bigg\}=:c(U).\ee
It remains to prove that $c(U)>0$. To this end we use \cite[Theorem 6]{DTGRF} which tells that the support of the random variable $f_\infty$ can be computed as the closure, in the $C^{\infty}$ topology, of the set of all the functions of the form $u\mapsto K_{\infty}(u, v)\lambda$, where $v$ ranges over $D$ {and $\lambda$ ranges in $\R^k$}. Using the explicit description of $K_{\infty}$, one can show that this space contains all the monomials $u_1^{\alpha_1}\cdots u_n^{\alpha_n}$ (see \cite[Lemma 3.2]{Ancona}) and, since polynomials are dense in the $C^{\infty}$--topology, the support of $f_\infty$ is the whole $C^{\infty}(D, \R^k)$. This means that for every open set $U\subseteq C^{\infty}(D,\R^k)$ the probability $c(U)$ is strictly positive. 
\end{proof}

\subsection{Compute an expectation at the global scale}\label{sec:step5}Given $\Sigma_0$ from \cref{sec:open}, consider the function $\nu:P_{n,d}^k\to \N$ defined by
$$\nu(p):=\textrm{number of connected components of $\Sigma^{W}(p)$ which are homeomorphic to $\Sigma_0$}.$$
Using the conclusions from \cref{sec:step4}, we show that there exist $\widetilde{c}(W)>0$ and $d(W)>0$ such that for all $d\geq d(W)$
\be\label{eq:integral0}\int_{\Po_{n,d}}\nu(p)\,\mu_d(\mathrm{d}p)\geq \widetilde{c}(W)d^n.\ee
(Notice that the integrand is bounded a.e. by \eqref{eq:bound}, since $\nu(p)\leq b_0(\Sigma^W(p))$.) In order to prove \eqref{eq:integral0} we introduce first some preliminary objects. 

First, by a doubling argument, for every $d\in \N$ we can find points $z_1, \ldots, z_{m_d}\in S^n$ with $m_d\geq c_3d^n$, $c_3>0$, such that for every $i\neq j$ the spherical Riemannian balls $B(z_i, d^{-1})$ and $B(z_j, d^{-1})$ do not intersect. 

Then, for every $z\in S^n$ we define the following set $E_{z, d}\subset P_{n,k}^d$:
$$E_{z, d}:=\bigg\{p\,\bigg|\, \textrm{$(B(z, d^{-1}), \Sigma^W(p))$ is diffeomorphic to $(D, \Sigma_0)$ }\bigg\}.$$
Notice that $E_{z,d}$ contains the set  of polynomials $p$ such that $f_{z,d}\in U_0$ where $U_0$ comes from \cref{sec:open}. 

For every $z\in S^n$ we also consider the following function $\nu_{z,d}:P_{n,d}^k\to \N$
\begin{align}\nu_{z,d}(p):=&\textrm{number of connected components of $\Sigma^W(p)$ which are entirely contained in $B(z, d^{-1})$} \\
&\textrm{and which are diffeomorphic to  $\Sigma_0$}.\end{align}

Notice that $\nu_{z,d}|_{E_{z, d}}\geq 1$. 
\begin{proof}[Proof of  \eqref{eq:integral0}] Using the above notation, we have
\begin{align}\int_{\Po_{n,d}}\nu(p)\,\mu_d(\mathrm{d}p)&\geq \int_{\Po_{n,d}}\sum_{j=1}^{m_d}\nu_{z_j, d}(p)\,\mu_d(\mathrm{d}p)=\sum_{j=1}^{m_d} \int_{\Po_{n,d}}\nu_{z_j, d}(p)\,\mu_d(\mathrm{d}p)\\
&\geq \sum_{j=1}^{m_d} \int_{E_{z_j, d}}\nu_{z_j, d}(p)\,\mu_d(\mathrm{d}p)\\
&\geq  \sum_{j=1}^{m_d} \int_{E_{z_j, d}}1\,\mu_d(\mathrm{d}p)=\sum_{j=1}^{m_d}\mu_d(E_{z_j, d})\\
&\geq \sum_{j=1}^{m_d} \mu_d\bigg\{f_{z_j,d}\in U^W\bigg\}=m_d\cdot\mu_d\bigg\{f_{z_1,d}\in U^W\bigg\}\\
&\geq c_3d^n\mu_d\bigg\{f_{z_1,d}\in U^W\bigg\}.
\end{align}
Taking the $\liminf_{d\to \infty}$ on both sides and using the limit from \eqref{eq:lowerlimit}, then \eqref{eq:integral0} follows (here $\widetilde{c}(W):=c_3\cdot c(U^{W})>0$).
\end{proof}
\begin{remark}Notice in particular that, since for every $p\in P_{n,d}^k$ we have the inequality $b(\Sigma^W(p))\geq b(\Sigma_0)\nu(p)$, then
\be\label{eq:integral1}\int_{P_{n,d}^k}b(\Sigma^W(p))\,\mu_d(\mathrm{d}p)\geq c(W)d^n.\ee
\end{remark}

\subsection{Conclude the argument}\label{sec:step6} Since the integral in \eqref{eq:integral0} is upper bounded by the maximum of $\nu(p)$, taken over the set of $p$ with $j^rp$ transversal to $W$ (this set is a full measure semialgebraic set in $P_{n,d}^k$), then for every $d$ large enough there exists $p$ with $j^rp$ transversal to $W$ and such that $\nu(p)\geq c_0 d^n$. For such a $p$ and for every $0\leq i\leq n-\textrm{codim}(W)$ one has
$$b_i(\Sigma^W(p))\geq \nu(p)b_i(\Sigma_0)\geq c_{2,i}(W)d^n,$$
where $c_{2,i}(W):=b_i(\Sigma_0)\widetilde{c}(W)$.
Since the jet of $p$ is transversal to $W$, the same property holds for all $\widetilde{p}$ in a neighborhood $U_d$ of $p$.

This concludes the proof of \cref{thm:main}, in its stronger form. (The statement as in \cref{thm:main} follows already from the case $i=0$.)
\section{Remarks and examples}\label{sec:examples}
\subsection{Hypersurfaces with rich topology}\label{remark:H16}
The total Betti numbers of a real hypersurface $\Sigma$ in a real algebraic variety $X$  is constrained by the total  Betti number of its complex part $\Sigma(\C)$ by Smith--Thom inequality:
$$b(\Sigma;\Z_2)\leq b(\Sigma(\C);\Z_2)=O(d^n).$$ 
The hypersurface is said \emph{maximal} it if attains the above inequality. Asymptotically maximal hypersurfaces (those with maximum possible total Betti number at leading order in the degree) exist in projective spaces \cite{ItenbergViro} and in toric
varieties \cite{Bertrand}.  From these results one can deduce similar results on the sphere, therefore there is no need to use the method from this paper to prove the existence of spherical hypersurfaces with rich topology. In general algebraic varieties, however, results of this type are not known. To our knowledge, in this context the first deterministic result using probabilistic ideas is \cite{Ancona}, where M. Ancona used the method outlined here to show that every real algebraic variety contains real algebraic hypersurfaces whose Betti numbers grow as the maximal possible order.  
\subsection{Complex bounds that \emph{cannot} be attained}One should be careful about expecting in general that real maximal objects are those whose sum of the Betti numbers equals the sum of the Betti numbers of their complex part. This is true, for example, for projective hypersurfaces (see \cref{remark:H16}) and for the critical points of a polynomial function on the sphere (see \cite{Kozhasov}). However, in general, this is not true. For instance, as noticed in \cite{RRS}, Klein’s formula combined with Pl\"ucker's formula imply that a real algebraic projective plane curve of degree $d$ has at most $d(d-2)$ real inflection points, whereas the number of complex inflection points is $3d(d-2)$. In other words, only a fraction (one third) of the total number of inflection points can be real. In view of this, since the proof of the bound \eqref{eq:bound} ultimately uses B\'ezout's Theorem (after reducing the problem to a stratified Morse Theory problem),  \cref{thm:main}  is in some sense the best that one can hope for. 

\subsection{Maximizing the constants}We believe that the constant $c_2(W)>0$ from \cref{thm:main} is far from optimal. For instance, if this method is applied for constructing  curves in the sphere $S^2$ with many ovals, i.e. with the choice $r=0$, $W=S^2\times \{0\}\subset J^0(S^2, \R)$, the constant is quite small: M. Nastasescu \cite{Nastasescu} has done numerical simulations for the expected number of ovals of a random curve with respect to our measure \eqref{eq:measure}, showing that
$$\frac{1}{d^2}\int_{P_{2,d}}b_0(\Sigma^{\{S^2\times \{0\}\}}(p))\mu_d(\mathrm{d} p)\approx 0.0195,$$
but the maximal curves have $\frac{(d-1)(d-2)}{2}+1=\frac{d^2}{2}+O(d)$ ovals. A possibility for increasing the constant is changing the probability measure.  We suspect that choosing the $L^2$--measure concentrated on top spherical harmonics should maximize (or at least increase) this  constant, but still we do not think that this method can produce maximal objects. (The exception to this statement is the case of zeroes of polynomials of one variable: spherical harmonics of the top degree in this case are linear combinations of $\cos (d\theta)$ and $\sin (d\theta)$ and therefore have the maximal possible number of zeroes.)

\subsection{How constructive is this method?}\label{sec:BW}
In the case of projective hypersurfaces, the construction of maximal objects typically uses (variations of) Viro's patchwork. If one tries to write a maximal hypersurface by putting random Gaussian coefficients in front of the (rescaled) monomial basis, the probability that she will get a hypersurface with rich topology decays exponentially as $d\to \infty$. More precisely, if we sample $p\in P_{n,d}$ according to the law\footnote{This is called a random Bombieri--Weyl polynomial, or Kostlan polynomial.}
$$p(x)=\sum_{|\alpha|=d}\xi_{\alpha} \left(\frac{d!}{\alpha_0!\cdot\alpha_n!}\right)^{1/2}x_0^{\alpha_0}\cdots x_n^{\alpha_n},$$
where $\{\xi_{\alpha}\}$ is a family of independent standard Gaussians, then \cite{DiattaLerario} implies that for every $c>0$ there exist $a_1, a_2>0$ such that
$$\mathbb{P}\bigg\{b(Z(p))\geq c d^n\bigg\}\leq a_1e^{-c_2 d}.$$
Similar statements hold for polynomial singularities \cite{BKL}.

In the general case of maps singularities we are not aware of any other method, other than the one proposed in the current paper, for producing maps with singularities of the maximal possible order. Moreover, the construction of the probability measure $\mu_d$ from the current paper gives a way to build polynomial maps with rich singularities with high probability, as follows. First recall that an equivalent way to define the probability distribution $\mu_d$ is by defining the random polynomial (i.e. a random variable with values in the space of polynomials) 
$$p(x):=\sum_{d-\ell \in 2\mathbb{N}}\sum_{j\in J_\ell}\xi_{\ell, j}\|x\|^{d-\ell}h_{\ell, j}\left(\frac{x}{\|x\|}\right),$$ 
where, for every $\ell$ such that $d-\ell$ is even, $\{h_{j, \ell}\}_{j\in J_\ell}$ is an $L^2(S^n)$--orthonormal basis for the space of harmonic polynomials of degree $\ell$, and $\{\xi_{\ell, j}\}_{\ell, j}$ is a family of independent standard Gaussian variables (see \cite{SHSP}). Then the integral \eqref{eq:integral1} equals
$$\int_{\Po_{n,d}}b_i(\Sigma^W(p))=\mathbb{E}b_i(\Sigma^W(p)).$$
Therefore, by putting random coefficients in front of the spherical harmonics basis, the probability of getting a map with rich singularities can be bounded from below. More precisely, letting $0<c_2(W)\leq c_1(W)$ be the constants from \eqref{eq:bound} and \cref{thm:main}, respectively, then for every $0\leq c\leq c_1(W)$, 
\be \label{eq:reversemarkov}\mathbb{P}\bigg\{b(\Sigma^W(p))\geq c d^n\bigg\}\geq 1-\frac{c_1(W)-c_2(W)}{c_1(W)-c},\ee
(which, of course, is interesting when $c$ is far from $c_1(W)$).
\begin{proof}[Proof of \eqref{eq:reversemarkov}]Recall Markov's inequality: for a non--negative random variable $\zeta\geq 0$ and for every $t> 0$,
\be\label{eq:markov}\mathbb{P}\left\{ \zeta\geq t\right\}\leq \frac{\mathbb{E}\zeta}{t}.\ee
Let now $\zeta:P_{n,d}^k\to [0, \infty)$ be the non--negative random variable $\zeta:=c_1(W)d^n-b(\Sigma^W(\cdot))$ (non--negativity a.e. follows from \eqref{eq:bound}). Then
\begin{align}\mathbb{P}\bigg\{b(\Sigma^W(p))\geq c d^n\bigg\}&=1-\mathbb{P}\bigg\{b(\Sigma^W(p))\leq c d^n\bigg\}=1-\mathbb{P}\bigg\{\zeta\geq (c_1(W)-c)d^n\bigg\}\\
&\geq 1-\frac{\mathbb{E}\zeta}{(c_1(W)-c)d^n}\quad \textrm{(by \eqref{eq:markov})}\\
&\geq 1-\frac{c_1(W)d^n-c_2(W)d^n}{(c_1(W)-c)d^n}\quad \textrm{(by \cref{thm:main})}\\
&= 1-\frac{c_1(W)-c_2(W)}{c_1(W)-c}.
\end{align}
\end{proof}
\begin{remark}The price that one needs to pay to get rich singularities in the spherical harmonics basis (compared to the monomial one) is having an expression for the spherical harmonics themselves. There is an interesting inductive construction for building spherical harmonics of degree $\ell$ on $S^n$ once one knows all the spherical harmonics of degree up to $\ell$ on the sphere $S^{n-1}$ (see \cite[Section 7.2]{SHSP}). In principle one can therefore start from the trigonometric functions (the harmonics on $S^1$) and build up all the $n$--dimensional harmonics, but she will need to know also the Gegenbauer polynomials. 
\end{remark}


%
%

\subsection{A variation on the argument: hypersurfaces with many umbilics}\label{sec:umbilics}The same scheme from \cref{sec:proof} can be used to prove the existence of other type of rich singularities. We discuss in this section the example the existence of algebraic hypersurfaces with many umbilics, as suggested by B. Shapiro.
To explain the result, let us fix some notation first.

Following \cite{multiple}, for every $n\in \N$ and for every $w=(w_1, \ldots, w_n)\in \N^n$ such that 
$$\mu(w):=\sum_{i=1}^n i w_i=n,$$ we denote by $\mathscr{C}_n^w\subset \mathrm{Sym}(n, \R)$ the set of real symmetric matrices with exactly $w_i$ eigenvalues of multiplicity $i$, for every $i=1, \ldots, n$. In this way, the set of symmetric matrices with multiple eigenvalues equals
\be\label{eq:stratify}\Delta_n:=\bigsqcup_{\{w\,|\, w_1<n\}}\mathscr{C}_{n}^w\ee
(its complement is $\mathrm{Sym}(n, \R)\setminus \Delta_n=\mathscr{C}_n^{(n, 0,\ldots, 0)}$).  The decomposition \eqref{eq:stratify} gives a semialgebraic stratification of $\Delta_n$ \cite{SV} whose strata have codimension \cite{Arnold}:
\be \label{eq:codimension}k(w):=\mathrm{codim}_{\mathrm{Sym}(n, \R)}\left(\mathscr{C}_n^w\right)=\sum_{i=1}^n\frac{(i-1)(i+2)}{2}w_i,\ee
Notice in particular that $\Delta_n$ has codimension $2$.

Let now $(M,g)$ be a smooth Riemannian manifold of dimension $n$ and $\Sigma\subset M$ be a smooth hypersurface. For every $z\in \Sigma$ let $\nu_z\in T_zM$ be a unit normal to $\Sigma$ (the unit normal is defined up to a sign). The \emph{second fundamental form} of $\Sigma$ at $z$ in the direction of $\nu_z$ is the quadratic form $h_{z}^\Sigma:T_z\Sigma\to \R$ defined by
\be\label{eq:second}h_z^{\Sigma}(v):=g_z(\nu_z, \nabla_vv),\quad v\in T_z\Sigma,\ee
where $\nabla$ is the Levi--Civita connection of $(M, g)$.  
For every $w=(w_1, \ldots, w_{n-1})\in \N^n$ such that $\mu(w)=n-1$ we denote now by $\Sigma^w$ the set of points in $\Sigma$ such that the symmetric matrix representing $h_z^\Sigma$ in some (and hence all) orthonormal basis for $T_z\Sigma$ belongs to $\mathscr{C}_{n-1}^w$ (this definition does not depend on the choice of the unit normal).
Notice that, if the second fundamental form of $\Sigma$ has the appropriate transversality property, the codimension of the set $\Sigma^w$ in $\Sigma$ is given by $k(w)$, as in \eqref{eq:codimension} with the sum only up to $n-1$. 
For a given $w=(w_1, \ldots, w_{n-1})\in \N^n$ such that $\mu(w)=n-1$ we will call the set $\Sigma^w$ the set of \emph{$w$--umbilics of $\Sigma$}.
(For example, when $n=3$, $\Sigma$ is a surface and the set $\Sigma^{(0, 1)}$ consists of the standard umbilics of $\Sigma$.)

We will be interested in umbilics of hypersurfaces defined by smooth functions $p:M\to \R$, this is why we introduce the following definition.
\begin{definition}Let $(M, g)$ be a smooth Riemannian manifold. For a smooth function $p:M\to \R$ such that the equation $p=0$ is regular on $M$ and for every $w\in \N^{n-1}$ we define $\Sigma^w(p)\subset M$ as the set of $w$--umbilics of the zero set of $p$. 
\end{definition}

In this case, it is useful to observe the following fact. 

\begin{lemma}\label{lemma:useful}Let $(M, g)$ be a smooth Riemannian manifold and let $p:M\to \R$ be a smooth function such that the equation $p=0$ is regular on $M$ and define $\Sigma:=\{p=0\}$ (a smooth hypersurface of $M$). Then for every $z\in \Sigma$ we have an identity of quadratic forms:
$$\|\grad_z (p)\| h_z^\Sigma=-\mathrm{he}(p)_z|_{T_{z}\Sigma},$$
where $\mathrm{he}(p)$ denotes the covariant Hessian and $\grad(p)$ the Riemannian gradient.
\end{lemma}
\begin{proof}Recall first the definition of the covariant Hessian of $p$:
\be \mathrm{he}(p)(v):=v(v p)-(\nabla_vv)p.\ee
When restricted to $u\in T_z\Sigma=\mathrm{ker}(D_zp)$, we have
$$\mathrm{he}(p)_z(u)=-((\nabla_uu)p)_z=-D_zp(\nabla_u u).$$
On the other hand, by definition \eqref{eq:second}, we have
$$h_z^\Sigma(u)=g_z(\nabla_u u, \nu(z))=g_z\left(\nabla_u u, \frac{\grad_z(p)}{\|\grad_z(p)\|}\right)=\frac{1}{\|\grad_z(p)\|}D_zp(\nabla_u u).$$
From this the conclusion follows.
\end{proof}

Clearly, the definition of $w$--umbilics of $\Sigma$ depends on the ambient Riemannian metric, but it is actually a conformal invariant. The following proposition is probably well--known to expert, but we were unable to locate it in the literature.

\begin{proposition}\label{propo:conformal}Let $(M, g)$ be a smooth $n$--dimensional Riemannian manifold and $\Sigma\subset M$ be a smooth hypersurface. For every $w=(w_1, \ldots, w_{n-1})\in \N^n$ such that $\mu(w)=n-1$ the set of $w$--umbilics of $\Sigma$ depends only on the conformal class of $g$.
\end{proposition}

\begin{proof}
{
Let $\f \in C^\infty(M,\R)$ and let us consider the conformally equivalent metric $\widetilde{g}:=e^{2\f} g$ on $M$. For any metric quantity $X$ associated with $g$, we will denote by $\widetilde{X}$
, the analogous quantity relative to $\widetilde{g}$. The relation between the Levi--Civita connections of the two metrics is given by
\be
\widetilde{\nabla}_VW=\nabla_VW+V(\f)W+W(\f)V-g(V,W)\grad (\f ).
\ee
(This follows from the Koszul identity.)
From this we deduce the relation between the two second fundamental forms of the hypersurface $\Sigma$. Let $V$ be a smooth vector field $V$ on $M$ with the property that $V(z)\in T_z\Sigma$ for all $z\in \Sigma$. Then, for any such $z\in \Sigma$ and $v=V(z)\in T_z\Sigma$, we have
\be 
\begin{aligned}
\widetilde h_z^{\Sigma}(v)&=\widetilde g_z(\widetilde \nu_z, \widetilde \nabla_VV|_z)
\\
&=
e^{2\f(z)} g_z\left(\frac{1}{e^{\f(z)}}\nu_z, \left(\nabla_VV+2V(\f)V-g(V,V)\grad (\f )\right)|_z\right)
\\
&=
e^{\f(z)} g_z\left( \nu_z, \nabla_vV-g_z(v,v)\grad_z (\f )\right)
\\
&=
e^{\f(z)} \left(
h^\Sigma_z(v)-g_z(v,v)D_z\f(\nu_z).
\right)
\end{aligned}
\ee
Let $v_1,\dots,v_{n-1}$ be a basis of $T_z\Sigma$ that is orthonormal for the metric $g$ and let $H$ be the matrix that represents the bilinear symmetric form associated to $h^\Sigma_z$, with respect to the chosen basis. Then, the vectors $\widetilde{v}_i:=e^{-\f(z)}v_i$ form an orthonormal basis for $\widetilde{g}$ and, by the above computation, the matrix $\widetilde H$ that represents $\widetilde{h}^\Sigma_z$ in the basis $\widetilde{v}_1,\dots,\widetilde{v}_{n-1}$ is
\be
\widetilde{H}=e^{-\f(z)}\left(H-D_z\f(\nu_z)\cdot \mathbf{1}_{n-1}
\right).
\ee
The eigenspaces of $\widetilde H$ coincide with those of $H$, thus $H\in \mathscr{C}_n^w$ if and only if $\widetilde H\in \mathscr{C}_n^w$. This yields that $z$ is a $w$-umbilic of $\Sigma$ in $(M,g)$ if and only if it is in $(M,\widetilde{g})$.
}
\end{proof}

Let now $z\in S^n\subset \R^{n+1}$ and denote by $\phi_z:\R^n\to S^n\setminus\{z\}$ the inverse of the stereographic projection from $z${, divided by two}. If $z=-e_0=(-1, 0, \ldots,0)$ this map  can be explicitly written as
\be \label{eq:stereographic}\phi_{-e_0}(y)={\frac12 \frac{1}{(1+\|y\|^2)}\left(\begin{array}{c}1-\|y\|^2\\2y\end{array}\right).}
\ee
The general case $z\neq -e_0$ can be obtained by composing \eqref{eq:stereographic} with an orthogonal transformation.
It is a well known fact (and immediate to check) that for every $z\in S^n$ the map 
$$\phi_z:(\R^n, g_{\R^n})\to (S^n\setminus \{z\}, g_{S^n}),$$
where $g_{\R^n}$ denotes the standard Riemannian structure on $\R^n$ and $g_{S^n}=g_{\R^{n+1}}|_{S^n}$ is the standard Riemannian structure on $S^n$ induced by the ambient Euclidean space, is a \emph{conformal map}.\footnote{Note that the stereographic projection from the center of the sphere to its tangent space \emph{is not} conformal (unless $n=1$).} In fact
\be\label{eq:conformal}\phi_z^*g_{S^n}=\frac{{1}}{(1+\|y\|^2)^2}\cdot g_{\R^n}.\ee

It follows from \eqref{eq:conformal} and \cref{propo:conformal}, that for every $z\in S^{n}$, if $p:S^n\to \R$ is smooth with $p=0$ regular,
$$\phi_z^{-1}\left(\Sigma^w(p)\right)=\Sigma^w(p\circ \phi_z),$$
i.e. the structure of the $w$--umbilics of $\{p=0\}$ (with respect to the Riemannian metric of the sphere) is the same as the structure of the $w$--umbilics of $\{\phi_z\circ p=0\}$ (with respect to the metric of $\R^n$).

Let now $p=(p_1, \ldots, p_k)\in P_{n,d}^k$. (We will need to work just in the case $k=1$, but we state the next property in greater generality.) We modify now the above definition \eqref{eq:rescaling} in order to study umbilical properties.  For every $z\in S^n$ we consider the smooth map $\widehat{f}_{z,d}:D\to \R$ defined by
\be \label{eq:rescaling2}\widehat{f}_{z,d}(v):=d^{-\frac{n}{2}}p\left(\phi_{z}(d^{-1} v)\right).\ee
As above, since we have defined a probability measure $\mu_d$ on the space of polynomials, the map $\widehat{f}_{z,d}:D\to \R^k$ (which depends on $p$) can be thought as a random variable with values in $C^{\infty}(D, \R^k)$. The following proposition substitutes \eqref{eq:lowerlimit}.

\begin{proposition}\label{propo:lower2}For every nonempty open set $U\subseteq C^{\infty}(D, \R^k)$ there exists ${c(U)}>0$ such that for every $z\in S^n$, 
\be \label{eq:lowerlimit2}\liminf_{d\to \infty}\mu_d\bigg\{\widehat{f}_{z,d}\in U\bigg\}\geq c(U).\ee
\end{proposition}
\begin{remark}In fact, from the proof it will follow that the number $c(U)>0$ is the same as in  \eqref{eq:defc(U)}.
\end{remark}
\begin{proof}
{We rehearse the proof of \eqref{eq:lowerlimit}. The only difference is that this time we have to study the covariance function of $\widehat{f}_d:=\widehat{f}_{z,d}$, which we denote by $\widehat{K}_{d}(x,y):D\times D\to \R^{k\times k}$. It follows that we have the following relation between the two random fields:
\be 
\widehat{f}_d(u)=f_d\left(d\cdot \mathrm{exp}_z^{-1}(\phi_z(d^{-1}u))\right)=f_d\left(d\cdot \tau (d^{-1}u)\right),
\ee
where we denote by $\tau\colon \R^n\to D(0,\pi)\subset \R^n$ the (deterministic) function defined as $\tau:=\mathrm{exp}_z^{-1}\circ \phi_z$. It is easy to check that this function is smooth, that $\tau(0)=0$ and that $d_0\tau=\mathrm{id}_{\R^n}$. From this, we conclude that as $d\to \infty$, the function $u\mapsto d\cdot \tau (d^{-1}u)$ converges in the $C^\infty(\R^n,\R^n)$-topology to the identity function $d_0\tau$. This, together with \cite{Hrmander1968TheSF}, implies that the covariance function $\widehat{K}_{d}$ converges in the $C^\infty$--topology to the function $\widehat{K}_\infty$, defined as in \eqref{eq:limitcovariance}.
}
\end{proof}

We are now in the position of proving the following result.

 \begin{theorem}\label{thm:umbilics}For every $w=(w_1, \ldots, w_{n-1})\in \N^{n-1}$ such that $\mu(w)=n-1$ and such that $k(w)\leq n-1$ there exists $d(w)\in \N$ and $\widehat{c}_2(w)>0$ such that for every $d\geq d(w)$ there exists an open set $U_d^w\subset \R[x_0, \ldots, x_n]_{(d)}$ such that for every $p\in U_d^w$ and for every $j=0, \ldots, n-1-k(w)$
$$b_j(\Sigma^w(p))\geq \widehat{c}_2(w) d^n.$$\end{theorem}
\begin{proof}The proof proceeds following the same steps as in \cref{sec:proof}, with some small modifications.

Given $w=(w_1, \ldots, w_{n-1})\in \N^{n-1}$ such that $\mu(w)=n-1$ and such that $k(w)\leq n-1$, denote by $W_0^w\subset J^2(n,1)$ the semialgebraic set
\be\label{eq:w}W_0^w:=\bigg\{(x, a, u, q)\in \R^n\times \R\times \R^n\times \mathrm{sym}^2(\R^n)\,\bigg|\, a=0, \,q|_{u^\perp}\in \mathscr{C}^w\bigg\}.\ee
(The notation ``$q|_{u^\perp}\in \mathscr{C}^w$'' means that the symmetric matrix representing $q|_{u^\perp}$ in some orthonormal basis for $u^\perp$ is in $\mathscr{C}^w$.) Then, because of \cref{lemma:useful}, for a smooth function $f:D\to \R$ with $f=0$ a regular equation, we have:
$$\Sigma^w(f)=j^2f^{-1}(W_0^w).$$

Then, as in \cref{sec:step1}, we can find a smooth function $f:D\to \R$ such that (i) $j^2f$ is transversal to $W_0^w$, (ii) $\Sigma_0:=j^2f^{-1}(W_0^w)\subset \mathrm{int}(D)$ and (iii) $b_i(j^2f^{-1}(W_0^w))\geq 1$ for every $0\leq i\leq n-1-k(w)$. As in \cref{sec:open}, this defines an open set $U_0$ in the $C^\infty$--topology such that every $g\in U_0$ also has these properties. 
We choose on the space of polynomials the same measure as in \cref{sec:step3}. 

Instead of \eqref{eq:rescaling}, we consider now the function $\widehat{f}_{z,d}$ defined in \eqref{eq:rescaling2}. To compute the limit probability of the open set $U_0$, instead of \eqref{eq:lowerlimit}, we use now \cref{propo:lower2} and obtain for every $z\in S^n$:
$$\liminf_{d\to \infty}\mu_d\bigg\{\widehat{f}_{z,d}\in U_0\bigg\}\geq c(U_0)>0.$$

Observe now that, since the map $\phi_z\left(d^{-1}\cdot \right)\colon D\to B(z,d^{-1})$ is conformal, \cref{propo:conformal} implies that the pairs
$$\left(D, \Sigma^w(\widehat{f}_{z,d})\right)\quad\textrm{and}\quad \bigg(B(z, d^{-1}), \Sigma^w(p)\cap B(z, d^{-1})\bigg) \quad \textrm{are diffeomorphic}.$$
(The fact that $\Sigma^w$ is a conformal invariant substitutes the requirement that the singularity is diffeomorphism invariant.)

The proof can proceed now as in \cref{sec:step5}:  we denote by $\nu:P_{n,d}\to \N$ the function
counting the number of connected components of $\Sigma^{w}(p)$ which are homeomorphic to $\Sigma_0$, 
and we show that $\int_{\Po_{n,d}}\nu(p)\,\mu_d(\mathrm{d}p)\geq \widetilde{c}(w)d^n$ for $d\geq d(w)$ and for some $\widetilde{c}(w)>0$.
Then, reasoning exactly as in \cref{sec:step6} concludes the argument.
\end{proof}
\begin{remark}Since the double cover $S^n\to \RP^n$ is a local isometry, the same result holds true for the $w$--umbilical points of the projective zero set of $p$.
\end{remark}

\begin{remark}Let $p\in P_{n,d}$ and pick $z\in S^n$ such that $p(z)\neq 0$. Let also $\varphi:=\phi_z^{-1}:S^{n}\setminus \{-z\}\to \R^n$ be the stereographic projection. For every $w=(w_1, \ldots, w_{n-1})\in \N^{n-1}$ such that $\mu(w)=n-1$ and such that $k(w)\leq n-1$, let $W^w\subset J^2(S^n, \R)$ be the subset $W^w:=\varphi^{-1}(W_0^w)$, where $W_0^w$ is defined in \eqref{eq:w}. Since $\varphi$ is semialgebraic, then so is $W^w$. Moreover, $j^2p$ is transversal to $W$ if and only if $j^2(p\circ\varphi)$ is transversal to $W^w$ and, by \cref{propo:conformal}, $j^2p^{-1}(W^w)$ equals the set of $w$--umbilics of $\{p=0\}$.
 Therefore we can use  \cite[Theorem 1]{MTTRPS} and conclude that there exists $\widehat{c}_1(w)>0$ such that for the generic polynomial $p:S^n\to\R$ of degree $d$, the sum of the Betti numbers of the $w$--umbilics of $\{p=0\}$ is bounded by:
\be\label{eq:boundw}b(\Sigma^w(p))\leq \widehat{c}_1(w)d^n.\ee
\end{remark}

\bibliographystyle{alpha}
\bibliography{TPMIRST}

\end{document}